\documentclass[a4paper,10pt]{article}
\usepackage{mathtext}
\usepackage[T1,T2A]{fontenc}
\usepackage[cp1251]{inputenc}
\usepackage[english]{babel}
\usepackage{amsmath}
\usepackage{amsfonts}
\usepackage{amssymb}
\usepackage{mathrsfs}
\usepackage{amsthm}
\usepackage{pb-diagram}
\usepackage{graphicx}
\usepackage{wrapfig}
\usepackage{mathrsfs}
\newcommand{\eps}{\varepsilon}
\newcommand{\Sp}{\mathfrak{S}}
\DeclareMathOperator{\vp}{v.p.}
\DeclareMathOperator{\OO}{O}

\DeclareMathOperator{\supp}{supp}

\DeclareMathOperator{\Ann}{Ann}

\DeclareMathOperator{\clos}{clos}
\DeclareMathOperator{\Ext}{Ext}
\newcommand{\R}{\mathcal{R}}
\newcommand{\Sph}{\mathrm{S}}
\newcommand{\Si}{_{_\Sigma}\!}
\newcommand{\df}{\buildrel\mathrm{def}\over=}

\newtheorem{Def}{Definition}
\newtheorem{Rem}{Remark}
\newtheorem{Conj}{Conjecture}

\renewcommand{\leq}{\leqslant}
\renewcommand{\geq}{\geqslant}

\newcommand{\scalprod}[2]{\langle{#1},{#2}\rangle}

\theoremstyle{plain}\newtheorem{Th}{Theorem}

\theoremstyle{plain}
\theoremstyle{plain}
\theoremstyle{plain}\newtheorem{Le}{Lemma}

\let\comma=,
{\catcode`\,=\active \gdef,{\relax\ifmmode\comma\else{\upshape\comma}\fi}

\tolerance = 400

\begin{document}
\title{Functions whose Fourier transform vanishes on a set}
\author{Dmitriy M. Stolyarov\thanks{Supported by RFBR grant N 14-01-00198.}}
\maketitle
\begin{abstract}
We study the subspaces of~$L_p(\mathbb{R}^d)$ that consist of functions whose Fourier transforms vanish on a smooth surface of codimension~$1$. We show that a subspace defined in such a manner coincides with the whole~$L_p$ space for~$p > \frac{2d}{d-1}$. We also prove density of smooth functions in such spaces when~$p < \frac{2d}{d-1}$ for specific cases of surfaces and give an equivalent definition in terms of differential operators.
\end{abstract}

\section{Definitions}
\begin{Def}\label{TheSpace}
Let~$\Sigma$ be a closed subset of~$\mathbb{R}^d$, let~$p \in [1,\infty]$. Define the space~$\Si L_p$ by the rule
\begin{equation*}
\Si L_p = \clos_{L_p}\Big(\big\{f \in \mathcal{S}(\mathbb{R}^d)\,\big|\; \forall \xi \in \Sigma\quad \hat{f}(\xi) = 0\big\}\Big).
\end{equation*}
\end{Def}
Here and in what follows, we denote the class of Schwartz functions by~$\mathcal{S}$ and the Fourier transform of~$f$ by~$\hat{f}$ or by~$\mathcal{F}[f]$.
There are many specific spaces that can be described as~$\Si L_p$ for various choices of~$\Sigma$, for example, the Paley--Wiener spaces ($\Sigma$ is the complement of a compact set,~$p=2$). In the present article, we restrict our attention to the case where~$\Sigma$ is a~$C^{\infty}$-smooth compact submanifold of~$\mathbb{R}^d$. Moreover, we assume that~$\dim \Sigma = d-1$. Our main example is the unit sphere~$\Sph^{d-1}$.

\begin{Def}\label{RestrictionOperator}
Let~$\Sigma$ be a closed subset of~$\mathbb{R}^d$. We call the operator~$\R_{\Sigma}\colon\mathcal{S}(\mathbb{R}^d) \to C(\Sigma)$ acting by the rule
\begin{equation*}
\R_{\Sigma}[f] = \hat{f}|_{_{\Sigma}}
\end{equation*}
the restriction operator associated with~$\Sigma$. We say that the statement~$\R(\Sigma,p,q)$ holds true if~$\R_{\Sigma}$ admits a continuous extension as an~$L_p(\mathbb{R}^d)\to L_q(\Sigma)$ operator. If~$\R(\Sigma,p,1)$ holds true, we say that~$\Sigma$ admits the~$L_p$ restriction.
\end{Def}
The definition needs two comments. First, to define~$L_q(\Sigma)$, we need to equip~$\Sigma$ with a measure. We always equip it with the natural Lebesgue measure, which we denote by~$\sigma$. Second, the number one in the definition of the~$L_p$ restriction property is not random:~$\R(\Sigma,p,q_1)$ leads to~$\R(\Sigma,p,q_2)$ provided~$q_1 > q_2$. We ask for the weakest natural restriction estimate. 

It is an extremly difficult problem to decide whether~$\R(\Sigma,p,q)$ holds true or not. The most well-studied case where~$\Sigma$ is convex and its curvature does not vanish is still far from being completely solved (of course, we omit the trivial case where~$\Sigma$ is a hyperplane and there is no restrictions except~$p=1$). We refer the reader to the books~\cite{Gr},~\cite{T} for the survey and the basics of the subject. We will not use the results of the restriction theory\footnote{However, some ideas of this theory may be found in Section~\ref{Sd=2}.}, we will usually assume that~$\Sigma$ admits the restriction property. 

One can also define the space of functions whose Fourier transform vanishes on~$\Sigma$ by the formula
\begin{equation}\label{TildaSpace}
\Si\tilde{L}_p = \big\{f \in L_p\,\big|\; \R_{\Sigma}[f] = 0 \big\}.
\end{equation}
Surely, such a definition needs~$\Sigma$ to admit the~$L_p$ restriction. 

There are two motivations for investigating such type spaces. The first one comes from the fact that the equation~$\R_{\Sph^{d-1}}[f] = 0$ may be considered as a natural Fredholm condition for the Helmholtz equation, see~\cite{G}. Second, such spaces arise if one considers inequalities in the style of Sobolev embedding theorems with differential operators. Let us consider an example from~\cite{S}. The author asked for which collections of parameters does the inequality
\begin{equation}\label{Sobolev}
\|f\|_{W_q^{\alpha,\beta}} \lesssim \|(\partial_1^k - \sigma \partial_2^l)f\|_{L_p}
\end{equation}
hold true for all functions~$f \in \mathcal{S}(\mathbb{R}^2)$ uniformly. The Sobolev potential norm on the left is defined by the formula
\begin{equation*}
\|f\|_{W_q^{\alpha,\beta}} = \Big\|\mathcal{F}^{-1}\big[\mathcal{F}[f](\xi,\eta)|\xi|^{\alpha}|\eta|^{\beta}\big] \Big\|_{L_q}.
\end{equation*}
The homogeneity considerations say that
\begin{equation}\label{HomogeneityCondition}
\frac{\alpha}{k} + \frac{\beta}{l} = 1 - \Big(\frac{1}{p} - \frac{1}{q}\Big)\Big(\frac{1}{k} + \frac{1}{l}\Big).
\end{equation}

Since~$f$ is smooth, the function~$(\partial_1^k - \sigma \partial_2^l)f$ lies in~$\Si L_p$, where
\begin{equation*}
\Sigma = \Big\{(\xi,\eta) \in \mathbb{R}^2\,\Big|\; (2\pi i \xi)^k = \sigma(2\pi i \eta)^l\Big\}.
\end{equation*}
The problem is interesting only when the operator~$(\partial_1^k - \sigma \partial_2^l)$ is not elliptic and~$k\ne l$, thus~$\Sigma$ is a curve in the plane. It does not satisfy the assumptions we impose on our surfaces in the present article (it is not compact and may not be  a smooth submanifold). 

However, after an application of the Littlewood--Paley theory and the homogeneity considerations, the problem may be localized in spectrum, i.e. we may assume that~$\hat{f}$ is supported in a small ball centered at some fixed point on~$\Sigma$ (here we omit the details). All in all, to prove inequality~\eqref{Sobolev}, one has to investigate the action of the Fourier multiplier with the symbol~$\vp \frac{1}{(2\pi i \xi)^k - (2\pi i \eta)^l}$ as an~$\Si L_p \rightarrow L_q$ operator. Note that the operator can be considered as a Bochner--Riesz type operator of order~$-1$. The continuity of such type operators on Lebesgue spaces is well studied in~$\mathbb{R}^2$, see~\cite{B}. In our case, it says that inequality~\eqref{Sobolev} holds true if~\eqref{HomogeneityCondition} is true,~$\frac{1}{p} - \frac{1}{q} \geq \frac23$,~$p < \frac43$, and~$q >4$. The classical Knapp example construction shows that the condition~$\frac{1}{p} - \frac{1}{q} \geq \frac23$ is necessary for~\eqref{Sobolev} to hold. Substitution of functions~$f$ that converge to the surface measure on~$\Sigma$ leads us to the necessity of the condition~$p < \frac43$. It was conjectured in~\cite{S} that the condition~$q > 4$ may be omitted if one considers inequality~\eqref{Sobolev}.

\begin{Th}[Easy corollary of Theorem~$5$ in~\cite{G}]\label{SobolevTh}
Suppose that~$1<p,q < \infty$. Inequality~\eqref{Sobolev} holds true if and only if~\eqref{HomogeneityCondition} holds true,~$\frac{1}{p} - \frac{1}{q} \geq \frac23$, and~$p < \frac43$.
\end{Th}

The paper~\cite{G} deals with the Bochner--Riesz operators of negative order, i.e.~$\Sigma = \Sph^{d-1}$. It is proved there that the continuity properties of these operators as~$\Si L_p \to L_q$ operators are better than their continuity properties as~$L_p \to L_q$ operators (i.e. there are many pairs~$(p,q)$ such that the Bochner--Riesz operator is not an~$L_p\to L_q$ operator, but it is an~$\Si L_p \to L_q$ operator). Using the same reasoning, one can prove the same results for general convex~$\Sigma$ with non-vanishing curvature, and then, using the Littlewood--Paley theory, derive Theorem~\ref{SobolevTh}. We note that the phenomenon has been noticed long time ago, see~\cite{GS}.

The paper~\cite{G} deals with~$\Si\tilde{L}_p$ spaces\footnote{Moreover,~$\Si\tilde{L}_p$ is defined with the help of~$\R(\Sigma,p,2)$. As we have said, the definition with~$\R(\Sigma,p,1)$ is more general.} instead of~$\Si L_p$. It is natural to conjecture that~$\Si\tilde{L}_p =  \Si\!\!L_p$ provided~$\Sigma$ admits the~$L_p$ restriction and~$\Si L_p=L_p$ provided it does not. In other words, the Schwartz functions are dense in~$\Si\tilde{L}_p$.

\begin{Th}\label{DensityWithRO}
Let~$p \in (1,\infty)$. Suppose that~$\Sph^{d-1}$ admits the~$L_p$ restriction. Then~$_{_{\Sph^{\scriptscriptstyle d-1}}}\!\tilde{L}_p = _{_{\Sph^{\scriptscriptstyle d-1}}}\!\!\!L_p$.
\end{Th} 

It is well known and easy to show that~$\R(\Sigma,p,q)$ cannot be true if~$p > \frac{2d}{d+1}$ (see, e.g.~\cite{T}). Thus,~$\Sigma$ cannot admit the~$L_p$ restriction for~$p > \frac{2d}{d+1}$. 
\begin{Th}\label{DensityWithoutMO}
For any~$p \in (\frac{2d}{d+1},\infty)$, we have~$\Si L_p=L_p$.
\end{Th}

The two theorems describe the situation for the case~$d=2, \Sigma = \Sph^1$ for all~$p \in (1,\infty)$ except for~$p=\frac43$. Using some specific geometric properties of the circumference, we can deal with the endpoint.
\begin{Th}\label{d=2}
Let~$\Sigma$ be a convex closed curve in the plane, suppose its curvature does not vanish. Then~$\Si L_{\frac43} = L_{\frac43}$.
\end{Th}

The proofs of all three theorems are based on duality, so we start with the description of the annihilators of our spaces in~$L_q$ ($q$ is the adjoint exponent,~$\frac{1}{p} + \frac{1}{q} = 1$) in Section~\ref{Annih}. After that we prove the main theorems in Sections~\ref{Rot},~\ref{dim}, and~\ref{Sd=2}. Finally, we give another definition of~$\Si L_p$ in Section~\ref{SPolynom} and prove its equivalence with the initial one. The last section contains several conjectures.

I am grateful to Michael Goldberg, Sergey Kislyakov, and  Andreas Seeger for discussions of the material of this article. 

\section{Description of annihilator}\label{Annih}
\begin{Def}\label{RealRestrictionOperator}
Define the operator~$\Pi_{\Sigma}\colon \mathcal{S}(\mathbb{R}^d) \to C^{\infty}(\Sigma)$ by the rule
\begin{equation*}
\Pi_{\Sigma}[f] = f|_{\Sigma}.
\end{equation*}
\end{Def}
\begin{Le}\label{Extension}
The operator~$\Pi_{\Sigma}$ admits right inverse.
\end{Le}
\begin{proof}
It suffices to extend any function~$\varphi \in C^{\infty}(\Sigma)$ to the whole space, and do this linearly. We denote the operator we are going to construct by~$\Ext$,~$\Ext\colon C^{\infty}(\Sigma)\to\mathcal{S}(\mathbb{R}^d)$. Its main property is the identity
\begin{equation*}
\Pi_{\Sigma}[\Ext[\varphi]] = \varphi,\quad \varphi \in C^{\infty}(\Sigma).
\end{equation*}
Let~$\{\zeta_n\}_n$ be a finite partition of unity on~$\Sigma$ that is subordinated to an atlas. It suffices to define the action of~$\Ext$ on the functions of the type~$\varphi\zeta_n$ since~$\varphi = \sum_n\varphi \zeta_n$. From now on, fix~$n$, let~$U_n$ be the corresponding chart neighborhood; we may assume that~$\varphi$ is supported in~$U_n$. Without loss of generality, we may also assume that~$\Sigma \cap U_n = \{(x,f(x))\,|\; x\in V\}$, where~$V \subset \mathbb{R}^{d-1} = \{x\in\mathbb{R}^d\mid x_d=0\}$ is a neighborhood of zero and~$f\colon V \to \mathbb{R}^d$ is a smooth function. Let~$\psi$ be a~$C_0^{\infty}(\mathbb{R})$ function such that~$\psi(0)=1$ and its support is sufficiently small. Define the extension of~$\varphi$ by the formula
\begin{equation*}
\Ext[\varphi](x,x_d) = \varphi(x,f(x))\psi(x_d - f(x)),\quad x=(x_1,x_2,\ldots,x_{d-1}) \in V.
\end{equation*}
First, this definition is smooth and linear in~$\varphi$, second,~$\Ext[\varphi]$ is supported in~$U_n$ provided the support of~$\psi$ is sufficiently small (and thus~$\Ext[\varphi]|_{\Sigma} = \varphi$ indeed). 
\end{proof}
We will use the adjont operator~$\Pi_{\Sigma}^*\colon (C^{\infty}(\Sigma))^\prime \to \mathcal{S}^\prime$ more often. 
\begin{Th}\label{Annihilator}
Let~$p \in [1,\infty)$. The annihilator of~$\Si L_p$ in~$L_q$,~$\frac{1}{p} + \frac{1}{q} = 1$, can be described as 
\begin{equation*}
\Ann_{L_q(\mathbb{R}^d)} \big(\,\Si L_p\big) = \Big\{g \in L_q(\mathbb{R}^d)\,\Big|\; \exists \zeta \in (C^{\infty}(\Sigma))^\prime\quad \hbox{such that}\quad \hat{g} = \Pi_{\Sigma}^*[\zeta]\Big\}.
\end{equation*}
\end{Th}
It is easy to see that the Fourier transform of a function~$g$ in the annihilator of~$\Si L_p$ is supported on~$\Sigma$. Theorem~\ref{Annihilator} says that the distribution~$\hat{g}$ is not only supported on~$\Sigma$, but can be treated as a distribution on~$\Sigma$, e.g. does not contain differentiations with respect to directions transversal to~$\Sigma$.
\begin{proof}
It suffices to show that every~$g \in L_q(\mathbb{R}^d)$ that annihilates~$\Si L_p$ can be represented as~$\mathcal{F}^{-1}\Pi^*_{\Sigma}[\zeta]$,~$\zeta \in (C^{\infty})^{\prime}(\Sigma)$. In other words, we need to construct~$\zeta$ satisfying the identity
\begin{equation}\label{ZetaIdentity}
\scalprod{\hat{g}}{f} = \scalprod{\zeta}{\Pi_{\Sigma}[f]},\quad \hbox{for all} \ f\in \mathcal{S}(\mathbb{R}^d).
\end{equation}
Consider the following diagram:
\catcode`\,=12
\begin{equation*}
\begin{diagram}
\node{\mathcal{S}(\mathbb{R}^d)}
\arrow[2]{e,t}{\Pi_{\Sigma}}
\arrow{se,b}{\hat{g}}
\node[2]{C^{\infty}(\Sigma)}
\arrow{sw,b}{?}\\
\node[2]{\mathbb{C}}
\end{diagram}\quad .
\end{equation*}
\catcode`\,=13
Then,~$\zeta$ corresponds to the arrow marked with a question. We will  construct~$\zeta$ as a map that makes the diagram commutative. First, since~$g$ annihilates~$\Si L_p$, we have the inclusion~$\ker \Pi_{\Sigma} \subset \ker \hat{g}$. Construct~$\zeta$ by the formula
\begin{equation*}
\zeta = \hat{g}\circ\Ext,
\end{equation*}
where~$\Ext$ is the right inverse of~$\Pi_{\Sigma}$ provided by Lemma~\ref{Extension}. Let us prove~\eqref{ZetaIdentity}, which is rewritten as
\begin{equation*}
\scalprod{\hat{g}}{f} = \scalprod{\hat{g}}{\Ext\big[\Pi_{\Sigma}[f]\big]}.
\end{equation*}
This is the same as to show that~$f-\Ext\big[\Pi_{\Sigma}[f]\big] \in \ker \hat{g}$. Since~$\ker \Pi_{\Sigma} \subset \ker \hat{g}$, it suffices to show that
\begin{equation*}
\Pi_{\Sigma}\Big[f-\Ext\big[\Pi_{\Sigma}[f]\big]\Big] = 0,
\end{equation*}
which is trivial since~$\Pi_{\Sigma}\circ \Ext$ is the identity operator.
\end{proof}
\section{Density of smooth functions with restriction operator}\label{Rot}
This section contains the proof of Theorem~\ref{DensityWithRO}. We start with a similar statement for the annihilator.
\begin{Le}\label{ApprAnLe}
Let~$p \in [1,\infty)$. The set
\begin{equation}\label{ApproximationAnn}
\Big\{g \in L_q(\mathbb{R}^d)\,\Big|\; \exists \zeta \in C^{\infty}(\Sph^{d-1})\quad \hbox{such that}\quad \hat{g} = \Pi_{\Sph^{d-1}}^*[\zeta]\Big\}
\end{equation}
is dense in~$\Ann_{L_q}\big(\,_{_{\Sph^{\scriptscriptstyle d-1}}}\! L_p\big)$.
\end{Le}
\begin{proof}
Consider the group~$\OO_d$ of rotations of~$\Sph^{d-1}$. Let~$\mu$ be the Haar measure on~$\OO_d$. Let~$\{\varphi_n\}_n$ be a smooth approximation of identity in~$\OO_d$. By this we mean that~$\{\varphi_n\}_n$ is a collection of~$C^{\infty}$-smooth functions on~$\OO_d$, each function~$\varphi_n$ is supported in a~$\frac{1}{n}$-neighborhood of the identity element, and~$\int_{\OO_d} \varphi_n\,d\mu = 1$ for all~$n$. It is easy to construct such a system of functions on the Euclidean space of appropriate dimension; then one can transfer the constructed functions to a neighborhood of the identity element in~$\OO_d$ using a chart mapping (and multiplying each function by a suitable constant to fulfill the condition~$\int_{\OO_d} \varphi_n\,d\mu = 1$).

Our aim is to approximate a function~$g \in \Ann_{L_q}\big(\,_{_{\Sph^{\scriptscriptstyle d-1}}}\! L_p\big)$ by functions belonging to the set~\eqref{ApproximationAnn}. By Theorem~\ref{Annihilator}, there exists a distribution~$\zeta \in \big(C^{\infty}(\Sph^{d-1})\big)^{\prime}$ such that~$\hat{g} = \Pi_{\Sph^{d-1}}^*[\zeta]$. The idea is to define the functions~$g_n$ and~$\zeta_n$ by the formulas
\begin{equation}\label{ZetaN}
\zeta_n = \int\limits_{\OO_d}\zeta_{T}\varphi_n(T)\,d\mu(T),\quad g_n = \mathcal{F}^{-1}\big[\Pi_{\Sph^{d-1}}^*[\zeta_n]\big].
\end{equation}
where~$\zeta_{T}$ is the rotation of~$\zeta$ by~$T$. Note that the integral~\eqref{ZetaN} should be understood in a special way (the integrand is distributional-valued). We define the averaging over~$\OO_d$ with the weight~$\varphi_n$ in the same way one defines the convolution of distributions with some fixed test function: the operator of averaging acts on~$C^{\infty}(\Sph^{d-1})$ and is formally self-adjoint, thus, one can define the averaging~\eqref{ZetaN} as a conjugate operator to the averaging of test functions.  

We note that the~$\zeta_n$ are smooth functions (it is instructive to consider the case~$d=2$, where the averaging~\eqref{ZetaN} turns into the classical convolution on the circumference; the proof for the general case is verbatim the reasoning for this classical case). Therefore, it remains to prove that~$g_n \to g$ in~$L_q(\mathbb{R}^d)$.

Since rotations commute with the Fourier transform, 
\begin{equation*}
g_n(x) = \int\limits_{\OO_d}g(Tx)\varphi_n(T)\,d\mu(T) \quad \hbox{for all } x \in \mathbb{R}^d.
\end{equation*}
We see that~$g_n$ is an~$L_q$-function. Using polar coordinates $\mathbb{R}^d \ni x = (r,\sigma) \in (0,\infty)\times \Sph^{d-1}$, we can write
\begin{equation}\label{qDifference}
\|g-g_n\|_{L_q}^q = \int\limits_{0}^{\infty}\bigg(\int\limits_{\Sph^{d-1}}|g(r,\sigma) - g_n(r,\sigma)|^q \,d\sigma\bigg)r^{d-1}\,dr.
\end{equation}
(There is an ambiguity in this formula: we denote by~$g$ and~$g_n$ both the initial functions and the ones they turn into after the change of variables; we also disregard the constant factor.) We see that~$g_n(r,\cdot)$ is the averaging of~$g(r,\cdot)$ with respect to~$\varphi_n\,d\mu$ for every fixed~$r$. By the properties of the approximation of identity,~$g_n(r,\cdot)\to g(r,\cdot)$ in~$L_q$ for every~$r$. Thus, we have a pointwise convergence of the integrands in the outer integral in~\eqref{qDifference}. Moreover, 
\begin{equation*}
\int\limits_{\Sph^{d-1}}|g(r,\sigma) - g_n(r,\sigma)|^q \,d\sigma \leq 2\int\limits_{\Sph^{d-1}}|g(r,\sigma)|^q,
\end{equation*} 
so, the integrands also posses a summable majorant. Therefore, by the Lebesgue convergence theorem,~$\|g_n - g\|_{L_q} \to 0$.
\end{proof}

\paragraph{Proof of Theorem~\ref{TildaSpace}.} Assume that~$_{_{\Sph^{\scriptscriptstyle d-1}}}\!\tilde{L}_p \nsubseteq  _{_{\Sph^{\scriptscriptstyle d-1}}}\!\!\! L_p$. By the Hahn--Banach theorem, there exist functions~$f$ and~$g$ such that
\begin{equation*}
f \in _{_{\Sph^{\scriptscriptstyle d-1}}}\!\!\!\tilde{L}_p, \quad g  \in \Ann_{L_q}\big(\,_{_{\Sph^{\scriptscriptstyle d-1}}}\! L_p\big), \quad \scalprod{f}{g} \ne 0.
\end{equation*}
Lemma~\ref{ApprAnLe} says that we may make~$g$ belong to the set~\eqref{ApproximationAnn}. Let~$\{f_n\}_n$ be a sequence of Schwartz functions approximating~$f$ in~$L_p(\mathbb{R}^d)$. On the one hand,
\begin{equation*}
|\scalprod{f_n}{g}| = |\scalprod{\hat{f}_n}{\Pi_{\Sph^{d-1}}^*[\zeta]}| = |\scalprod{\Pi_{\Sph^{d-1}}[\hat{f}_n]}{\zeta}| \leq \|\R_{\Sph^{d-1}}[f_n]\|_{L_1(\Sph^{d-1})}\|\zeta\|_{L_{\infty}} \lesssim \|f-f_n\|_{L_p}\|\zeta\|_{L_{\infty}} \to 0
\end{equation*}
since we have assumed that~$\Sph^{d-1}$ admits the~$L_p$ restriction and~$\R_{\Sph^{d-1}}[f] = 0$. On the other hand, ~$\scalprod{f_n}{g} \to \scalprod{f}{g} \ne 0$. A contradiction.\qed
\begin{Rem}
It was noticed by Michael Goldberg that one can actually omit the usage of Theorem~\textup{\ref{Annihilator}} in the reasoning above, i.e. can avoid passing to duals and directly construct the approximation of a function in~$_{_{\Sph^{\scriptscriptstyle d-1}}}\!\tilde{L}_p$ using the same averaging and properties of the extention operator built in Lemma~\textup{\ref{Extension}}.
\end{Rem}

\section{Density of smooth functions without restriction operator}\label{dim}
This section contains the proof of Theorem~\ref{DensityWithoutMO}.
\begin{Def}\label{Dimension}
Let~$\mu$ be a complex measure of locally bounded variation on~$\mathbb{R}^d$. The number
\begin{equation*}
\dim \mu \df \inf \Big\{\alpha\,\Big|\;\exists\, F \ \hbox{--- Borel set,}\ \mu(F) \ne 0, \dim F \leq \alpha\Big\}
\end{equation*}
is called the lower Hausdorff dimension of~$\mu$.
\end{Def}
We will use only the Hausdorff dimension for sets and measures, so we do not emphasize this. We will also call the lower dimension of measure simply the dimension.  The proof of Theorem~\ref{DensityWithoutMO} relies on Lemma~$2$ in~\cite{SW}. We will need a simplified version stated below. In this lemma,~$\varphi$ is a hat function, i.e. a~$C_0^{\infty}(\mathbb{R}^d)$-function that is non-negative, radial, non-increasing (as a function of radius), and equals one in a neighborhood of zero.
\begin{Le}[Simplified Lemma~$2$ in~\cite{SW}]\label{Frostman}
Let~$\mu$ be a complex measure of finite variation on~$\mathbb{R}^d$. If for all~$\xi\in\mathbb{R}^d$ and~$r \in (0,1)$
\begin{equation}\label{FrostmanCondition}
\bigg|\int_{\mathbb{R}^d}\varphi(r^{-1}\xi+\eta)\,d\mu(\xi)\bigg| \lesssim r^{\alpha},
\end{equation}
then~$\dim \mu \geq \alpha$.
\end{Le}
The classical Frostman lemma\footnote{The easier part of the Frostman lemma.} says that if~$\mu$ is non-negative and satisfies the inequality~$\mu(B_r(x)) \lesssim r^{\alpha}$ for all balls~$B_r(\xi)$ uniformly, then~$\dim \mu \geq \alpha$. Lemma~\ref{Frostman} deals with a more general case of a signed (or what is the same, complex) measure. We note without proof that~\eqref{FrostmanCondition} does not imply~$|\mu|(B_r(\xi)) \lesssim r^{\alpha}$. 

\paragraph{Proof of Theorem~\ref{DensityWithoutMO}.} By Theorem~\ref{Annihilator}, it suffices to prove that there does not exist a non-zero distribution~$\zeta \in (C^{\infty}(\Sigma))^{\prime}$ such that~$\mathcal{F}^{-1}\big[\Pi_{\Sigma}^*[\zeta]\big] \in L_q$,~$q < \frac{2d}{d-1}$. Assume the contrary, let~$\zeta$ be such a distribution. We will come to a contradiction in two steps: first we will show that~$\zeta \in L_2(\Sigma)$, second, we will apply Lemma~\ref{Frostman} to the measure~$\zeta\,d\sigma$ (this is a measure since~$\zeta \in L_2(\Sigma)$). But before that we note that the problem is local. Indeed, if we multiply~$\zeta$ by a smooth function supported in a neighborhood of a point on~$\Sigma$, we do not harm the condition~$\mathcal{F}^{-1}\big[\Pi_{\Sigma}^*[\zeta]\big] \in L_q$. Thus, we may assume that~$\zeta$ is supported in arbitrarily  small neighborhood~$U_s$ of some point~$s \in \Sigma$.

We rewrite the condition~$\mathcal{F}^{-1}\big[\Pi_{\Sigma}^*[\zeta]\big] \in L_q$ in a more convenient form. In the formula below, we identify~$\mathbb{R}^{d-1}$ with some fixed linear hyperplane in~$\mathbb{R}^d$:
\begin{equation*}
\int\limits_{\mathbb{R}^d}\Big|\mathcal{F}^{-1}\big[\Pi_{\Sigma}^*[\zeta]\big](x)\Big|^q\,dx = \int\limits_{\OO_d}\int\limits_{\mathbb{R}^{d-1}} \Big|\mathcal{F}^{-1}\big[\Pi_{\Sigma}^*[\zeta]\big](T\tilde{x})\Big|^q|\tilde{x}|\,d\tilde{x} \,d \mu(T),
\end{equation*}
here~$\mu$ stands for the Haar measure on~$\OO_d$ (we disregard a constant factor in this formula). Since the value on the left is finite, 
\begin{equation}\label{Slice}
\int\limits_{\mathbb{R}^{d-1}} \Big|\mathcal{F}^{-1}\big[\Pi_{\Sigma}^*[\zeta]\big](T\tilde{x})\Big|^q|\tilde{x}|\,d\tilde{x} < \infty
\end{equation}
for almost all~$T \in \OO_d$.  Therefore, there exists an element~$T \in \OO_d$ such that the hyperplane~$T\mathbb{R}^{d-1}$ is not collinear with the normals to~$\Sigma$ at the points in~$U_s$ and the corresponding integral~\eqref{Slice} is finite (we recall that~$U_s$ is small). We apply H\"older's inequality:
\begin{equation}\label{SquareSliceBound}
\int\limits_{|\tilde{x}| \geq 1} \Big|\mathcal{F}^{-1}\big[\Pi_{\Sigma}^*[\zeta]\big](T\tilde{x})\Big|^2\,d\tilde{x} \leq
\Bigg(\int\limits_{|\tilde{x}| \geq 1} \Big|\mathcal{F}^{-1}\big[\Pi_{\Sigma}^*[\zeta]\big](T\tilde{x})\Big|^q|\tilde{x}|\,d\tilde{x}\Bigg)^{\frac{2}{q}}\Bigg(\int\limits_{|\tilde{x}| \geq 1} |\tilde{x}|^{-\frac{2}{q-2}}d\tilde{x}\Bigg)^{\frac{q-2}{q}}.
\end{equation}
We note that~$\frac{2}{q-2} > d-1$ when~$q < \frac{2d}{d-1}$, so, the function~$\mathcal{F}^{-1}\big[\Pi_{\Sigma}^*[\zeta]\big](T\tilde{x})$ belongs to~$L_2(\mathbb{R}^{d-1})$. 

Since~$T\mathbb{R}^{d-1}$ is not collinear with the normals to~$\Sigma\cap U_s$,  one can represent the surface~$\Sigma\cap U_s$ as a graph of some function~$h\colon T\mathbb{R}^{d-1} \to \mathbb{R}$ in the orthonormal coordinates of the hyperplane~$T\mathbb{R}^{d-1}$ (here~$V_s \subset T\mathbb{R}^{d-1}$ is a neighborhood of zero):
\begin{equation*}
\Sigma\cap U_s = \{(t,h(t))\mid t \in V_s\}.
\end{equation*}  
The mapping~$t \to (t,h(t))$ is a diffeomorphism onto its image, therefore, there exists a distribution~$\zeta_T \in (C_0^{\infty}(V_s))^{\prime}$, whose image under this diffeomorphism is~$\zeta$. In the chosen coordinates, we have:
\begin{equation*}
\mathcal{F}^{-1}\big[\Pi_{\Sigma}^*[\zeta]\big] (t,\underline{t}) = \scalprod{\zeta_T}{e^{-2\pi i \big(\scalprod{t}{\cdot} + \underline{t}h(\cdot)\big)}},\quad t\in T\mathbb{R}^{d-1}, \,\,\underline{t} \in \mathbb{R}.
\end{equation*}
In particular,~$\mathcal{F}^{-1}\big[\Pi_{\Sigma}^*[\zeta]\big] (t,0)$ is the~$(d-1)$-dimensional inverse Fourier transform of~$\zeta_T$.  By~\eqref{SquareSliceBound}, the function~$t\mapsto\mathcal{F}^{-1}[\Pi_{\Sigma}^*[\zeta]](t,0)$ is square summable. Thus, by the Plancherel theorem,~$\zeta_T \in L_2$. Applying the inverse diffeomorphism, we see that~$\zeta$ \emph{is a square summable function with respect to the Lebesgue measure on~$\Sigma$}.

We are going to apply Lemma~\ref{Frostman} to the measure~$\nu =\zeta\,d\sigma$ to show that its dimension is strictly greater than~$d-1$ (obviously, this will give us the desired contradiction). It is convenient to use the following potential:
\begin{equation*}
I(\nu,\alpha)\df\int\limits_{\mathbb{R}^d} \big|\mathcal{F}^{-1}[\nu](x)\big|^2\big(1+|x|\big)^{\alpha-d}\,dx,
\end{equation*}
where~$\alpha \in (0,d)$. The inequality
\begin{equation*}
\int\limits_{\mathbb{R}^d} \big|\mathcal{F}^{-1}[\nu](x)\big|^2\big(1+|x|\big)^{\alpha-d}\,dx \leq  \bigg(\,\,\int\limits_{\mathbb{R}^d} \big|\mathcal{F}^{-1}[\nu](x)\big|^q\,dx\bigg)^{\frac{2}{q}}\bigg(\,\,\int\limits_{\mathbb{R}^d}\big(1+|x|\big)^{\frac{q(\alpha-d)}{q-2}}\,dx\bigg)^{\frac{q-2}{q}}
\end{equation*}
shows that~$I(\nu,\alpha)$ is finite for some~$\alpha > d-1$ (since~$\frac{q}{q-2} > d$ if~$2 \leq q < \frac{2d}{d-1}$; the case~$q < 2$ is obvious). It remains to show that the assumptions of Lemma~\ref{Frostman} follow from the condition~$I(\nu,\alpha) < \infty$. We use the notation~$\varphi_r(\xi) = \varphi(r^{-1}\xi)$. First, 
\begin{equation*}
\big|\scalprod{\nu}{\varphi_r(\cdot + \eta)}\big| \leq \int\limits_{\mathbb{R}^d}\big|\mathcal{F}^{-1}[\nu](x)\mathcal{F}^{-1}[\varphi_{r}](x)\big|\,dx \leq I(\nu,\alpha)^{\frac12}\Big(\int\limits_{\mathbb{R}^d}\big|\mathcal{F}^{-1}[\varphi_r](x)\big|^2 (1+|x|^2)^{\frac{d-\alpha}{2}}\,dx\Big)^{\frac12}.
\end{equation*}
Second,
\begin{multline*}
\int\limits_{\mathbb{R}^d}\Big|\mathcal{F}^{-1}[\varphi_r](x)\Big|^2\Big(1 + |x|^2\Big)^{\frac{d-\alpha}{2}}\,dx = \int\limits_{\mathbb{R}^d}\Big|r^{d}\mathcal{F}^{-1}[\varphi](rx)\Big|^2\Big(1 + |x|^2\Big)^{\frac{d-\alpha}{2}}\,dx \leq \\
\int\limits_{\mathbb{R}^d}\Big|r^{d}\mathcal{F}^{-1}[\varphi](rx)\Big|^2\Big(\frac{1}{r^2} + \frac{|rx|^2}{r^2}\Big)^{\frac{d-\alpha}{2}}r^{-d}\,d(rx) =r^{\alpha}\int\limits_{\mathbb{R}^d}\Big|\mathcal{F}^{-1}[\varphi](z)\Big|^2\Big(1 + |z|^2\Big)^{\frac{d-\alpha}{2}}\,dz,
\end{multline*}
provided~$r \leq 1$.\qed

\section{Case~$d=2$,~$p=\frac43$}\label{Sd=2}
In this section, we prove Theorem~\ref{d=2}. 
\paragraph{Proof of Theorem~\ref{d=2}.} By Theorem~\ref{Annihilator}, it suffices to prove that there does not exists a non-zero distribution~$\zeta \in (C^{\infty})^{\prime}(\Sigma)$ such that~$\mathcal{F}^{-1}\big[\Pi_{\Sigma}^*[\zeta]\big] \in L_4(\mathbb{R}^2)$. We may assume that~$\zeta$ is supported in a small ball~$U$ centered at a point on~$\Sigma$ (let it be the origin). It is convenient to parameterize~$\Sigma$ by the line:
\begin{equation*}
\Sigma\cap U = \big\{(\xi,h(\xi))\in \mathbb{R}^2\,\big|\,\, \xi \in (-\eps,\eps)\big\}.
\end{equation*}  
Since rotations commute with the Fourier transform, we may assume that~$h'(0)=0$. By our assumptions,~$h$ is convex. We also assume that~$h''$ is almost constant on~$(-\eps,\eps)$.

The condition~$\mathcal{F}^{-1}\big[\Pi_{\Sigma}^*[\zeta]\big] \in L_4(\mathbb{R}^2)$ can be rewritten as
\begin{equation}\label{ConvolutionCondition}
\Pi_{\Sigma}^*[\zeta]*\Pi_{\Sigma}^*[\zeta] \in L_2(\mathbb{R}^2).
\end{equation}
Let~$\Phi \in C_0^{\infty}\big(U\big)$, then the function~$\tilde{\Phi} \in C_0^{\infty}(\mathbb{R}^2)$ is defined as
\begin{equation}\label{PhiandTilde}
\tilde{\Phi}(s,t) = \Phi(s+t,h(s) + h(t)). 
\end{equation}
In the light of this definition, the action of the convolution on a test function is written as
\begin{equation}\label{ConvolutionFormula}
\scalprod{\Pi_{\Sigma}^*[\zeta]*\Pi_{\Sigma}^*[\zeta]}{\Phi} = \scalprod{\zeta}{\scalprod{\zeta}{\tilde{\Phi}}}.
\end{equation} 
Let us study the change of variables~$(\xi,\eta) = \big(s+t,h(s) + h(t)\big)$ in detail. The mapping~$(s,t) \to (\xi,\eta)$ is injective on the set~$s < t$ and~$C^{\infty}$-smooth there. Its image is the set~$(\Sigma\cap U) + (\Sigma\cap U)$. So, it allows us to make the change of variables:
\begin{equation}\label{ChangeOfVariables}
\iint\limits_{(-\eps,\eps)^2} \tilde{\Phi}(s,t)|h'(s) - h'(t)|\,ds\,dt = 2\!\!\!\!\!\!\!\!\!\!\iint\limits_{(\Sigma\cap U) + (\Sigma\cap U)} \!\!\!\!\! \Phi(\xi,\eta)\,d\xi\,d\eta.
\end{equation}


Let~$\varphi \in C_0^{\infty}\big((-\eps,\eps)\big)$ be a function. Then, by formula~\eqref{ConvolutionFormula}, 
\begin{equation*}
|\scalprod{\zeta}{\varphi}|^2 = \Big|\scalprod{\zeta}{\scalprod{\zeta}{\tilde{\Phi}}}\Big| = \Big|\scalprod{\Pi_{\Sigma}^*[\zeta]*\Pi_{\Sigma}^*[\zeta]}{\Phi}\Big| \leq \Big\|\Pi_{\Sigma}^*[\zeta]*\Pi_{\Sigma}^*[\zeta]\Big\|_{L_2}\big\|\Phi\big\|_{L_2},
\end{equation*}
where~$\tilde{\Phi}(s,t) = \varphi(s)\varphi(t)$ (and the function~$\Phi$ is recovered by formula~\eqref{PhiandTilde}). Using formula~\eqref{ChangeOfVariables}, we obtain
\begin{equation}\label{EstimateInL2}
\|\Phi\|_{L_2}^2 = \frac12\iint\limits_{(-\eps,\eps)^2} |\varphi(s)|^2|\varphi(t)|^2|h'(s) - h'(t)|\,ds\,dt \lesssim \|\varphi\|^4_{L_2(\mathbb{R})}.
\end{equation}
So,
$|\scalprod{\zeta}{\varphi}| \lesssim \|\varphi\|_{L_2(\mathbb{R})}$.
Thus,\emph{~$\zeta$ is a square summable function}.

Since~$\zeta$ is a summable function, the convolution can be computed directly using formulas~~\eqref{ConvolutionFormula} and~\eqref{ChangeOfVariables},
\begin{equation*}
\Pi_{\Sigma}^*[\zeta]*\Pi_{\Sigma}^*[\zeta](\xi,\eta) = \frac{\zeta\big(s,h(s)\big)\zeta\big(t,h(t)\big)}{2|h'(s)-h'(t)|},\quad s = s(\xi,\eta), t=t(\xi,\eta).
\end{equation*}
We are going to prove that the function on the right-hand side does not belong  to $L_2(\mathbb{R}^2)$. We may assume that the set~$\{s \mid |\zeta(s,h(s))| \geq 1\}$ has positive measure. Let~$I$ be an interval such  that
\begin{equation}\label{DensityLebesguePoint}
\Big|\{s\mid |\zeta(s,h(s)) |\geq 1\} \cap I\Big| \geq 0.9 |I|
\end{equation}
(such an interval exists around any Lebesgue point of the set~$\{s \mid |\zeta(s,h(s))| \geq 1\}$). Consider the set~$S_N \subset\mathbb{R}^2$ of the points~$(\xi,\eta)$ for which~$s(\xi,\eta),t(\xi,\eta) \in I$ and~$ \frac{1}{N} \leq |s-t| \leq \frac{1.1}{N}$. This set lies in a~$N^{-2}$-neighborhood of~$\Sigma$, and by the condition that~$\Sigma$ has non-vanishing curvature,~$|S_N|\asymp N^{-2}$. Moreover,
\begin{equation*}
|h'(s) - h'(t)| \asymp N^{-1}, \quad (\xi,\eta) \in S_N.
\end{equation*} 
It is not hard to see that 
\begin{equation}\label{Bound}
\bigg|\Big\{(\xi,\eta) \in S_N\,\Big|\; |\zeta(s,h(s))| \geq 1, |\zeta(t,h(t))| \geq 1\Big\}\bigg| \geq \frac12 |S_N|.
\end{equation} 
Indeed,~\eqref{DensityLebesguePoint} and the fact that~$|h'(s) - h'(t)|$ is almost constant on~$S_N$ (we have assumed that~$h''$ is almost constant) gives
\begin{equation*}
\Big|\Big\{(\xi,\eta) \in S_N \,\Big|\; |\zeta(s,h(s))| \leq 1 \Big\}\Big| \leq \frac{1}{5}|S_N|, \quad \Big|\Big\{(\xi,\eta) \in S_N \,\Big|\; |\zeta(t,h(t))| \leq 1 \Big\}\Big| \leq \frac{1}{5}|S_N|,
\end{equation*}
which leads to~\eqref{Bound}.

Therefore, by~\eqref{Bound}, the function~$(\xi,\eta)\mapsto\frac{\zeta(s,h(s))\zeta(t,h(t))}{|h'(s) - h'(t)|}$ is not less than~$N$ on the set of measure at least~$\frac12|S_N|\asymp N^{-2}$. This contradicts~\eqref{ConvolutionCondition}, because the sets~$S_N$ are disjoint.\qed

\section{One more definition of~$\Si L_p$}\label{SPolynom}
In this section we restrict our attention to a more specific class of surfaces. Let~$\Sp$ be a polynomial in~$d$ variables that grows at infinity. Let~$\Sigma$ be its zero set, i.e.
\begin{equation}\label{Polynomial_surface}
\Sigma =\{\xi \in \mathbb{R}^d\mid \Sp(\xi) = 0\}.
\end{equation} 
We assume that~$\nabla \Sp$ does not vanish on~$\Sigma$. Thus,~$\Sigma$ satisfies the conditions we imposed on it before. The aim of this section is to give an alternative definition of~$\Si L_p$ in this special case.
\begin{Th}\label{Polynom}
If~$\Sp$ is a polynomial in~$d$ variables that grows at infinity,~$\Sigma$ is given by~\eqref{Polynomial_surface}, and~$\nabla \Sp$ does not vanish on~$\Sigma$, then
\begin{equation*}
\Si L_p = \clos_{L_p}\bigg(\Big\{\Sp\Big(\frac{\partial}{2\pi i}\Big) [g] \,\Big|\; g \in \mathcal{S}(\mathbb{R}^d) \Big\}\bigg).
\end{equation*} 
\end{Th}
The notation~$\Sp(\partial)$ denotes the differential polynomial, e.g. the case of~$\Sigma = \Sph^{d-1}$ corresponding to~$\Sp(\xi) = \sum_{1}^d\xi_i^2 - 1$ results in the Helmholtz operator:
\begin{equation*}
\Sp\Big(\frac{\partial}{2\pi i}\Big)[g] = -\Big(\frac{\Delta}{4\pi^2} + 1\Big)[g].
\end{equation*}
We need a simple lemma about division of smooth functions.

\begin{Le}\label{DivisionLemma}
		Let~$P$ be a polynomial in~$d$ variables. Suppose that~$\varphi \in C_0^{\infty}(\mathbb{R}^d)$ vanishes on its zero set~$\{\xi\mid P(\xi)=0\}$ and~$\nabla P$ does not vanish on this set. In such a case,~$\frac{\varphi}{P} \in C_0^{\infty}(\mathbb{R}^d)$.
		\end{Le}
		\begin{proof}
		We may assume that~$\varphi$ is supported in a small ball whose center~$\xi_0$ lies on the surface~$\{\xi \in \mathbb{R}^d\mid P(\xi)=0\}$. We may also assume that~$\partial_d P$ does not vanish on~$\supp \varphi$.
		
		Let us first consider the easiest case~$d=1$ and~$P(\xi) = \xi$. It is clear that the function~$\frac{\varphi(\xi)}{\xi}$ is infinitely differentiable outside zero. Let us prove that it is~$k$ times differentiable at zero. Consider the order~$k+1$ Taylor polynomial of~$\varphi$ at~$0$:
		\begin{equation*}
		\varphi(\xi) = \sum\limits_{i=1}^{k+1}\frac{\varphi^{(i)}(0)\xi^i}{i!} + O(\xi^{k+2}).
		\end{equation*}  		
		Surely, 
		\begin{equation*}
		\frac{\varphi(\xi)}{\xi} = \sum\limits_{i=0}^{k}\frac{\varphi^{(i+1)}(0)\xi^i}{(i+1)!} + O(\xi^{k+1}),
		\end{equation*}  
		which shows that~$\frac{\varphi(\xi)}{\xi}$ is~$k$ times differentiable at zero. 
				
		The general case  may be derived from the trivial one-dimensional case considered with the help of the change of variables~$\xi \to \eta$,~$\eta_j = \xi_j$ for~$j=1,2,\ldots,d-1$ and~$\eta_d = P(\xi)$. Due to the condition~$\partial_d P \ne 0$, such a change of variables is a~$C^{\infty}$-homeomorphism of~$\supp \varphi$ onto its image. Thus, this change of variables takes~$\varphi$ to a function~$\tilde{\varphi}\in C_0^{\infty}(\mathbb{R}^d)$ and~$P$ to a simple polynomial~$\eta_d$. Applying the already considered statement for the case~$d=1$ to~$\tilde{\varphi}$, we obtain that~$\frac{\tilde{\varphi}}{\eta_d} \in C_0^{\infty}(\mathbb{R}^d)$. Consequently,~$\frac{\varphi}{P} \in C_0^{\infty}(\mathbb{R}^d)$.
		\end{proof}
		\begin{Rem}
		In fact, we did not use that~$P$ is a polynomial. It could have been a~$C^{\infty}$ function. Lemma~\textup{\ref{DivisionLemma}} is a common knowledge, for example, it may be found in~\textup{\cite{MS}}.
		\end{Rem}
		
\paragraph{Proof of Theorem~\ref{Polynom}.}
By Definition~\ref{TheSpace}, it suffices to prove that
\begin{equation*}
\big\{f \in \mathcal{S}(\mathbb{R}^d)\,\big|\; \forall \xi \in \Sigma\quad \hat{f}(\xi) = 0\big\} = \Big\{\Sp\Big(\frac{\partial}{2\pi i}\Big) [g] \,\Big|\; g \in \mathcal{S}(\mathbb{R}^d) \Big\}.
\end{equation*}
Clearly, the latter set belongs to the former. Thus, we need to prove that for every~$f \in \mathcal{S}(\mathbb{R}^d)$ such that~$\hat{f}$ vanishes on~$\Sigma$, there exists~$g \in \mathcal{S}(\mathbb{R}^d)$ such that~$\Sp\Big(\frac{\partial}{2\pi i}\Big) [g] = f$. In the Fourier transform world, this statement looks like Lemma~\ref{DivisionLemma}. The only difference is that the class~$C_0^{\infty}$ is replaced by the Schwartz class. To overcome this difficulty, one has to multiply the function~$\hat{f} = \varphi$ by a~$C_0^{\infty}$ function that equals one in a neighborhood of~$\Sigma$, and use the fact that~$\Sp$ grows at infinity.\qed

The definition of~$\Si L_p$ provided by Theorem~\ref{Polynom} can be naturally generalized.
\begin{Def}\label{NewHardy}
Let~$k \in \mathbb{N}$, define the space~$\Si L_p^k$ by the formula
\begin{equation*}
\Si L_p^k = \clos_{L_p}\bigg(\Big\{\Sp^k\Big(\frac{\partial}{2\pi i}\Big) [g] \,\Big|\; g \in \mathcal{S}(\mathbb{R}^d) \Big\}\bigg).
\end{equation*} 
\end{Def}
\begin{Conj}
For every~$d$,~$p$, and~$\Sp$ fixed, there exists~$l$ such that~$\Si L_p^l = \Si \!\! L_p^{l+1} = \Si \!\! L_p^{l+2} = \ldots$, i.e. the chain of these spaces stabilzes.
\end{Conj}

\section{Conjectures}
The study of~$\Si L_p$ spaces is far from being complete. We list several conjectures the author is interested in.

First, it is natural to replace the sphere in Theorem~\ref{DensityWithRO} by a more general manifold. It suffices to prove the conjecture below.
\begin{Conj}
Let~$p \in [1,\infty)$ be such that~$\Sigma$ admits the~$L_p$ restriction property. The set
\begin{equation*}
\Big\{g \in L_q(\mathbb{R}^d)\,\Big|\; \exists \zeta \in C^{\infty}(\Sigma)\quad \hbox{such that}\quad \hat{g} = \Pi_{\Sigma}^*[\zeta]\Big\}
\end{equation*}
is dense in~$\Ann_{L_q}(\Si L_p)$.
\end{Conj}
We have proved the conjecture for the case~$\Sigma = \Sph^{d-1}$, our reasoning heavily relied on the rotational symmetry of the sphere. A similar reasoning works for the case of a more general quadratic surface.

The proofs of Theorems~\ref{DensityWithoutMO} and~\ref{d=2} are based on duality. In fact, we prove an uncertainty principle (in the sense of the book~\cite{HJ}): a Fourier transform of a small function (in our case small means fastly decaying) cannot be small (here small means generated by a distribution ``depending on less number of variables''). Our proof of Theorem~\ref{DensityWithoutMO} uses two properties of~$\Sigma$: the continuity of the normal and its dimension. The author believes that the structure of the distribution~$\zeta$ may be omitted.
\begin{Conj}
There does not exists a non-zero distribution~$\mathfrak{Z} \in \mathcal{S}^{\prime}(\mathbb{R}^d)$ such that~$\mathcal{F}^{-1} [\mathfrak{Z}]\in L_q$,~$ q <\frac{2d}{d-1}$, and~$\supp \mathfrak{Z} \subset \Sigma$. In the case where~$\Sigma$ is convex, there does not exist a distribution~$\mathfrak{Z}$ such that~$\mathcal{F}^{-1} [\mathfrak{Z}]\in L_{\frac{2d}{d-1}}$ and~$\supp \mathfrak{Z} \subset \Sigma$.
\end{Conj}

It was proved in~\cite{G} that the continuity properties of the Bochner--Riesz operator of negative order improve when~$L_p$ is replaced by~$\Si L_p$. It seems that the spaces~$\Si L_p^k$ introduced in Definition~\ref{NewHardy} may be thought of as ``Hardy classes'' for Bochner--Riesz operators.
\begin{Conj}
There exists~$l \in \mathbb{N}$ depending on~$\lambda > 0$ such that the Bochner--Riesz operator of order~$\lambda$ acts from~$_{_{\Sph^{\scriptscriptstyle d-1}}}\!L_p^k$ to~$L_p$ continuously for every~$p$ and every~$k > l$. 
\end{Conj}

More generally, it is interesting to get more information about the continuity properties of the Bochner--Riesz operators acting between~$\Si L_p^k$ spaces. We believe that interpolation techniques might be helpful here. In the conjecture below, the rectangular brackets denote the complex interpolation method, the usual brackets denote the real method. 
\begin{Conj}
The scale~$\Si L_p$ is stable under interpolation, i.e.
\begin{equation*}
[\Si L_p,\, \Si L_q]_{\theta} = \Si\!\! L_r, \quad (\Si L_p,\;\! \Si L_q)_{\theta,r} = \Si\!\! L_r ,\quad \frac{\theta}{p} + \frac{1-\theta}{q} = \frac{1}{r}, \quad 1 \leq p,q \leq \infty.
\end{equation*}
\end{Conj}

Dmitriy M. Stolyarov

Michigan State University, Department of Mathematics;

P. L. Chebyshev Research Laboratory, St. Petersburg State University;

St. Petersburg Department of Steklov Mathematical Institute, Russian Academy of Sciences (PDMI RAS).

\medskip

dms at pdmi dot ras dot ru, dms at math dot msu dot edu.

\medskip

http://www.chebyshev.spb.ru/DmitriyStolyarov.


\begin{thebibliography}{99999} 

\bibitem{B} J.-G.~Bak, \emph{Sharp estimates for the Bochner--Riesz operator of negative order in~$\mathbb{R}^2$}, Proc. AMS {\bf 125}:7 (1997), 1977--1986.

\bibitem{G} M. Goldberg, \emph{The Helmholtz Equation with $L^p$ data and Bochner--Riesz Multiplier}, to appear in Math. Res. Lett.

\bibitem{GS} M. Goldberg, W. Schlag, \emph{A limiting absorption principle for the Schrodinger equation with $L^p$ potentials}, Intl. Math. Res. Not. {\bf 75} (2004), 4049--4071.

\bibitem{Gr} L. Grafakos, \emph{Modern Fourier analysis}, Sec. ed., Springer, 2009.

\bibitem{HJ} V. Havin, B. J\"oricke, \emph{The Uncertainty principle in harmonic analysis}, Springer, 1994.

\bibitem{MS} M. A. Mostow, S. Shnider, \emph{Joint continuity of division of smooth functions. \textup{I:} uniform Lojasiewicz estimates}, Tran. AMS {\bf 292}:2 (1985), 573--583.

\bibitem{S} D. M. Stolyarov, \emph{Bilinear embedding theorems for differential operators in $\mathbb{R}^2$}, Zap. nauchn. sem. POMI {\bf 424} (2014), 210--235 (Russian); English translation: J. Math. Sc. (New York) {\bf 206}:5 (2015), 792--807.

\bibitem{SW} D. M. Stolyarov, M. Wojciechowski, \emph{Dimension of gradient measures}, C. R. Math. {\bf 352}:10 (2014), 791--795.

\bibitem{T} T. Tao, \emph{Recent progress on the Restriction conjecture}, http://arxiv.org/abs/math/0311181.

\end{thebibliography}
\end{document}